\documentclass[10pt]{amsart}
\usepackage{amsthm,amsfonts, amssymb, amscd}
\newenvironment{dem}{\begin{proof}[\bf Proof]}{\end{proof}}
\newtheorem{thm}{Theorem}
\newtheorem{lemma}[thm]{Lemma}
\newtheorem{prop}[thm]{Proposition}

\theoremstyle{definition}

\newtheorem{remark}[thm]{\bf Remark}

\numberwithin{subsection}{section}
\newcommand{\A}{\mathbb A}

\newcommand{\Hcal}{\mathcal H}

\newcommand{\Pro}{\mathbb P}

\newcommand{\Z}{\mathbb Z}

\newcommand{\Oc}{\mathcal O}

\newcommand{\Gm}{\mathbb G_{\textbf{m}}}

\DeclareMathOperator{\Pic}{Pic}
\DeclareMathOperator{\Aut}{Aut}
\DeclareMathOperator{\Char}{char}
\DeclareMathOperator{\GL}{GL}

\DeclareMathOperator{\Spec}{Spec}
\DeclareMathOperator{\Proj}{Proj}
\DeclareMathOperator{\Sym}{Sym}

\begin{document}
\numberwithin{thm}{section}

\title[The Chow ring of the stack of hyperelliptic
curves]
{The integral Chow ring of the stack of hyperelliptic
curves of even genus}
\author{Dan Edidin and Damiano Fulghesu}
\address{Department of Mathematics, University of Missouri, Columbia, MO 65211}
\email{edidin@math.missouri.edu, damiano@math.missouri.edu}
\maketitle

\begin{abstract} 
Let $g$ be an even positive integer.
In this paper we compute the integral Chow ring of the stack
of smooth hyperelliptic curves of genus $g$.
\end{abstract}

\section{Introduction}
A natural question, inspired by
David Mumford's classic paper {\em Toward an enumerative geometry of
  the moduli space of curves} \cite{Mum:83}, is to compute the integral
Chow rings of the stacks ${\mathcal M}_g$ and $\overline{{\mathcal
    M}}_g$ of smooth and stable curves of a given genus $g$.  In
\cite{EdGr:98} the integral Chow rings of the stacks of elliptic
curves ${\mathcal M}_{1,1}$ and $\overline{{\mathcal M}}_{1,1}$, were
computed. In an appendix to the same paper Vistoli  \cite{Vis:98} computed the Chow
ring of ${\mathcal M}_2$. However, for $g\geq 3$ almost nothing is
known. The only positive result is the computation of $\Pic({\mathcal
  M}_g)$ and $\Pic(\overline{{\mathcal M}}_g)$ by Arabarello and
Cornalba \cite{ArCo:87} using Harer's computation of the second
homology group of the mapping class group \cite{Har:83}. Even
rationally, the Chow rings of ${\mathcal M}_g$ have been computed only
up to to $g =5$ \cite{Fab:90a, Fab:90b, Iza:95}.

In this paper we focus our attention to hyperelliptic
curves and obtain a result valid for all (even) genera.
To be precise, let ${\mathcal H}_g$ denote the stack of smooth hyperelliptic
curves of even genus $g$ defined over a field $k$.
\begin{thm} \label{thm.main}
Assume that $\Char k = 0$
or $\Char k > 2g$. Then 
$$A^*({\mathcal H}_g) = \Z[c_1,c_2]/(2(2g+1)c_1, g(g-1)c_1^2 - 4g(g+1)c_2).$$
\end{thm}
\begin{remark}
When $g =2$, every curve is hyperelliptic and our theorem recovers
Vistoli's presentation for $A^*({\mathcal M}_2)$. Theorem
\ref{thm.main} also recovers Arsie and Vistoli's result \cite{ArVi:04}
that $\Pic({\mathcal H}_g)$ is cyclic of order $2(2g+1)$.
\end{remark}

\begin{remark} Note that the generators $c_1, c_2$ of $A^*({\mathcal H}_g)$  are not in general 
tautological classes. Gorchinskiy and Viviani \cite{GoVi:06} observed that
$\Pic({\mathcal H}_g)$ is not generated by
$\lambda$, the first Chern class of the Hodge bundle, when $g\equiv 0 \mod 4$.
In Section \ref{sec.taut} we describe a natural vector bundle on ${\mathcal H}_g$
whose Chern classes generate the Chow ring.
\end{remark}

This theorem is quite surprising since it implies that integral Chow
rings of stacks of hyperelliptic curves of even genus have a
remarkably simple structure. This is in marked contrast to the situation
in topology, where results about the additive structure of the
cohomology of the hyperelliptic mapping class group are quite
complicated (cf. \cite{BCP:01}).

Our techniques are purely algebraic and make essential use of results
of Arsie and Vistoli \cite{ArVi:04}.  As observed in \cite[Example
3.5]{ArVi:04} the stack ${\mathcal H}_g$ may be identified as
the stack of double covers of $\Pro^1$ branched at $2g+2$ points. Theorem
4.1 of \cite{ArVi:04} then implies that if $g$ is even ${\mathcal
H}_g$ is the quotient (stack) by an action of $\GL_2$ on the open set
in ${\mathbb A}^{2g+3}$ corresponding to homogeneous binary forms of
degree $2g+2$ with distinct roots.  Using a basic result in
equivariant intersection theory we then identify $A^*({\mathcal H}_g)$
with the $\GL_2$-equivariant Chow ring of this open set.

From this description it immediately follows from
the basic exact sequence in Chow groups that $A^*({\mathcal H}_g)$
is a quotient of $A^*_{\GL_2} \A^{2g+3} = \Z[c_1, c_2]$. Because
inclusion of an open set does not induce a surjection on
cohomology, the calculation in topology is much more difficult.
Note that this situation is reversed 
for ${\mathcal M}_g$ where there seems to be no algebraic way to
get information about the Chow groups for general $g$.

Following a strategy employed by Vistoli in his calculation 
of $A^*({\mathcal M}_2)$ the computation of the equivariant Chow ring
can be reduced to the calculation of the $\GL_2$-equivariant
Chow ring of $\Pro^{2g+2} \smallsetminus \Delta_1$, where $\Delta_1$
is the locus of forms divisible by a square. Rather than
tackle this problem directly, as Vistoli did for $g=2$, we reduce to the maximal
torus $T \subset \GL_2$. The technical heart of this
paper is the proof  that, for any $N$,
$A^*_T(\Pro(\Sym^N E^*) \smallsetminus \Delta_1)$ is the quotient 
of $A^*_T(\Pro(\Sym^N E^*))$ by an ideal generated by two
classes (Proposition \ref{prop.deg}), where $E$ is the defining representation
of $\GL_2$. Because $\GL_2$ is {\em special}\footnote{This means that every
$\GL_n$-torsor is locally trivial in the Zariski topology}
we can then recover the $\GL_2$-equivariant Chow ring.

\section{Background on group actions and equivariant intersection theory}
\subsection{Group actions and representations}
Let $k$ be a field and let $G$ be an algebraic group over $k$.
If $V$ is a representation of $G$ over $k$, then, when the context is clear,
we refer to the scheme $\Spec(\Sym V^*)$ as $V$. With this convention
if $x \in \Spec(\Sym V^*)$ is a $k$-valued point corresponding to 
a vector $v \in V$ then for any $g \in G(k)$,
$gx$ is the $k$-valued point corresponding to the vector $gv$.
If the action of $G$ on $V$ commutes with the diagonal action of ${\mathbb G}_m$
on $V$ then there is an induced action of $G$ on $\Pro(V) := \Proj(\Sym V^*)$.
As a consequence, if $d$ is positive integer then the global sections $H^0(\Pro(V),{\mathcal O}_{\Pro(V)}(d))$ correspond to the $G$-module $\Sym^d V^*$.

\subsection{Some general facts about equivariant Chow groups}
Equivariant Chow groups are defined in the paper \cite{EdGr:98}.
We briefly 
recall some basic facts and notation that we will use in our computation.

Let $G$ be a linear algebraic group defined over a field $k$.  For any
algebraic space $X$ we denote the direct sum of the equivariant Chow
groups by $A_*^G(X)$.  If $X$ is smooth then there is a product
structure on equivariant Chow groups and we denote the equivariant
Chow ring by $A^*_G(X)$.  Following standard notation, we denote the
equivariant Chow ring of a point by $A^*_G$.  Flat pullback $X \to
\Spec k$ makes the equivariant Chow groups $A_*^G(X)$ into an
$A^*_G$-module. When $X$ is smooth the equivariant Chow ring
$A^*_G(X)$ becomes an $A^*_G$ algebra.

The relation between equivariant Chow rings and Chow rings
of quotient stacks is given by the following result.
\begin{prop}\cite[Propositions 17, 19]{EdGr:98} \label{prop.chowstack}
Let $G$ be an algebraic group and let $X$ be a smooth $G$-space
and let ${\mathcal F} = [X/G]$ be the quotient stack.
Then the equivariant Chow ring $A^*_G(X)$
is independent of the presentation for ${\mathcal F}$ and may
be identified with the integral Chow ring of ${\mathcal F}$.
\end{prop}

\medskip

\subsection{Equivariant Chow rings for $\GL_n$ actions}
Let $T = \Gm^n$ be a maximal torus.  Because $\GL_n$ is special the
restriction homomorphism $A^*_{\GL_n} \to A^*_T$ is injective and the
image consists of the classes invariant under the action of the Weyl
group $W(T,\GL_n) = S_n$.  Hence, we may view $A^*_{\GL_n}$ as a
subalgebra of $A^*_T$.  More generally \cite[Proposition 3.6]{EdGr:98}
or \cite[Theorem 6.7]{Bri:97a}) imply that if $X$ is an algebraic space then the
restriction map $A_*^{\GL_n}(X) \to A_*^T(X)$ is an injective
homomorphism of $A^*_{\GL_n}$-modules.  Likewise, if $X$ is smooth,
the restriction map $A^*_{\GL_n} X \to A^*_T X$ is an injective
homomorphism of $A^*_{\GL_n}$-algebras.  In both cases the images
consist of elements which are invariant under the natural action of
the Weyl group.  If $A_*^{\GL_n}(X)$ is a flat $A^*_{\GL_n}$ module
then \cite[Theorem 6.7]{Bri:97a} also implies that $A_*^{\GL_n}(X) =
(A_*^T(X))^{S_n}$.
\footnote{In \cite{EdGr:98} it was incorrectly claimed that if $G$ is
special and $T$ is a maximal torus then for all $G$-spaces $X$,
$A^*_G(X) = (A^*_T(X))^W$ where $W=W(G,T)$ is the Weyl group. The
second author is grateful to Reyer Sjamaar for pointing out this
error.}

The following result which we
will be very useful in the proof of Theorem \ref{thm.main}.
\begin{prop} \label{prop.tsummand}
Let $G$ be a special algebraic group and let $T \subset G$
be a maximal torus. If $X$ is a smooth $G$-space then $A^*_G(X)$
is (non-canonically) a summand in the $A^*_G(X) $-module $A^*_T(X)$.
\end{prop}
\begin{proof}  
Fix an integer $i$ and let $V$ be a representation of $G$ which
contains an open set $U \subset V$ on which $G$ acts freely such that
$V\smallsetminus U$ has codimension $>i$.  Then we can
identify $A^i_G(X) = A^i(X \times_G U)$ and $A^i_T(X) = A^i(X \times_T
U)$. The restriction map $A^i_G(X) \to A^i_T(X)$ corresponds to the
flat pullback $\pi \colon X \times_T U \to X \times_G U$. Let $B
\supset T$ be a Borel subgroup. Since $B/T$ is isomorphic to  affine space we may
also identify $A^i_T(X)$ with $A^i(X\times_B U)$.  Since $G$ is
special the $G/B$ bundle $p \colon X \times_B U \to X \times_G U$ is
locally trivial in the Zariski topology.  Hence by \cite[Lemma
7]{EdGr:97}, $A^*(X \times_G U)$ is (non-canonically) a summand in
$A^*(X \times_B U)$.
\end{proof}

\subsection{Chern classes and equivariant Chow rings
of projective spaces} 

If $V$ is a representation of $G$, then $V$
defines a $G$-equivariant vector bundle over $\Spec k$. Consequently a
representation $V$ of rank $r$ has {\it Chern classes} $c_1(V), \ldots ,
c_r(V) \in A^*_G$. If $X$ is a smooth algebraic space then we will view
the Chern classes as elements of the equivariant Chow ring $A^*_G(X)$
via the pullback $A^*_G \to A^*_G(X)$.

Now let  $E$ be the defining representation of $\GL_n$. The 
total character of the $T$-module $E$ decomposes into a sum of linearly
independent characters 
$\lambda_1 +  \lambda_2 + \ldots \lambda_n$ and
we get $A^*_T = \Z[t_1, \ldots t_n]$ where $t_i = c_1(\lambda_i)$.
The Weyl group $S_n$ acts on $A^*_T$ by permuting the $t_i$'s and 
as result $A^*_{\GL_n} = \Z[c_1, \ldots , c_n]$ where 
$c_i = c_i(E)$ is the $i$-th elementary symmetric polynomial in $t_1, \ldots , t_n$
\cite{EdGr:97}.

If $V$ is a representation of rank $r$ of
$\GL_n$
then then 
the total character of the $T$-module $V$ decomposes as sum of characters $\mu_1 +
\ldots \mu_r$. Let
$l_i = c_1(\mu_i)$. 
We refer to the classes $l_1, \ldots l_r$ as the {\it Chern roots}
of $V$ and view them as elements in $A^*_T X$. Any symmetric
polynomial in the Chern roots is an element of $A^*_{\GL_n} X$.

Let $V$ be an $(r+1)$-dimensional representation of $\GL_n$.
Since the action of $\GL_n$ commutes with the diagonal action of ${\mathbb G}_m$
on $V$ there is an induced action of $\GL_n$ on $\Pro(V)$
a canonical $\GL_n$-linearization of the sheaf
${\mathcal O}_{\Pro(V)}(1)$.

The following easy lemma is proved for torus actions in 
\cite[Section 3.3]{EdGr:98} 
and follows in general from the projective
bundle theorem \cite[Example 8.3.4]{Ful:84}.
\begin{lemma} \label{lem.projspace}
The $\GL_n$ (resp. $T$) equivariant Chow ring of $\Pro(V)$ has the following presentation.
$$A^*_{\GL_n}(\Pro(V))= 
A^*_{\GL_n}[\xi]/(\xi^{r+1} + C_1 \xi^n + \ldots C_{r+1})$$
and
$$A^*_T(\Pro(V)) = A^*_T[\xi]/\left(\prod_{i =1}^r (\xi + l_i)\right)$$
where $\xi = c_1(\Oc_{\Pro(V)}(1))$,  $C_1, \ldots C_r \in 
A^*_{\GL_n}$ (resp. $l_1, \ldots , l_r$)  
are the equivariant Chern classes (resp. Chern roots) 
of the representation $V$.
\end{lemma}

\subsection{$T$-equivariant fundamental classes of invariant hypersurfaces}
Let $T$ be a torus and let $V$ be a finite dimensional $T$-module and let
${\mathcal O}_{\Pro(V)}(1)$ have the canonical linearization induced by the action
of $T$ on $V$.
\begin{lemma} \label{lem.Thypersurface}
Let $H \subset \Pro(V)$ be a $T$-invariant hypersurface defined by a 
homogeneous form $f \in \Sym^d(V^*)$ such that $z \cdot f = \chi^{-1}(z)f$
for some character $\chi \colon T \to {\mathbb G}_m$.
Then in $A^*_T(\Pro(V))$,
$$[H]_T = c_1^T({\mathcal O}_{\Pro(V)}(d)) + c_1(\chi)$$
\end{lemma}
\begin{proof}
On the invariant affine $a_i \neq 0$, the ideal sheaf of
$H$ is generated by the $T$-eigenfunction $f/x_j$. Hence
${\mathcal O}(-H) = {\mathcal O}(-d) \otimes \chi^{-1}$
where ${\mathcal O}(-d)$ is given its canonical $T$-linearization.
\end{proof}

Let $V$ be a two-dimensional representation of $T$ and choose
coordinates so that $T$ acts by $z\cdot(a_0,a_1) = (\chi_{0}(z)a_0, \chi_{1}(z)a_1)$.
Let $\Pro^1 = \Pro(V)$ and let
$\Delta \subset \Pro^1 \times \Pro^1$ be the diagonal.
Then $\Delta$ is a $T$-invariant hypersurface for the diagonal action
of $T$ on $\Pro^1 \times \Pro^1$. 
\begin{lemma} \label{lem.p1diag}
In $A^*_T(\Pro^1)$ we have the identity
$$[\Delta_{\Pro^1 \times \Pro^1} ] = [0 \times \Pro^1] + [\Pro^1 \times 0] + c_1(\chi_1 \chi_0^{-1})$$
\end{lemma}
\begin{proof}
Let $x_0,x_1$ (resp. $y_0,y_1$) be coordinate functions on the first (resp. second) copy of $\Pro^1$.
Then $\Delta_{\Pro^1 \times \Pro^1}$ is defined by the homogeneous quadratic form $x_0y_1 - x_1 y_0$
while $[0 \times \Pro^1] + [\Pro^1 \times 0]$ is defined by the homogeneous quadratic form
$x_0y_0$. Now $x_0y_1 - x_1 y_0$ is in the $\chi_0^{-1}\chi_1^{-1}$-eigenspace of $V^* \otimes V^*$
and $x_0y_0$ while is in the $\chi_0^{-2}$ eigenspace. The formula now follows from
the same argument used in Lemma \ref{lem.Thypersurface}.
\end{proof}

\section{Arsie and Vistoli's presentation $\Hcal_g$ as quotient stack}
To reduce the computation of the $A^*(\Hcal_g)$ to a calculation 
in equivariant intersection theory we recall the presentation for
$\Hcal_g$ for $g$ even given by Arsie and Vistoli in \cite{ArVi:04}.

Let $k$ be a field of characteristic not equal
2 and assume that $g$ is an even integer. Let $E$ be the defining
representation of $\GL_2$ and let ${\mathcal D} \colon \GL_2 \to
{\mathbb G}_m$ be the determinant.

\begin{thm}\cite[Corollary 4.7]{ArVi:04} \label{presHg}
If $g$ is even the stack $\Hcal_g$ is isomorphic to the quotient
$$
\left[ \left( \Sym^{2g+2} E^* \otimes {\mathcal D}^{\otimes g} 
\smallsetminus \Delta_1 \right) / \GL_2 \right]
$$
where $\Delta_1$ is the closed subvariety of singular forms.
\end{thm}

By Proposition \ref{prop.chowstack} the Chow ring of $\Hcal_g$
may identified with the equivariant Chow ring
$$
A^*_{\GL_2}(\Sym^{2g+2}E^* \otimes {\mathcal D} \smallsetminus \Delta_1).
$$
and the remainder of the paper is devoted to performing this
computation.

Let $\Pro^{2g+2} = \Pro(\Sym^{2g+2}E^*)$ be the projective space of
forms of degree $2g+2$ and again let $\Delta_1$ be the hypersurface corresponding
to singular forms. Following the argument of Vistoli \cite{Vis:98}
we may reduce to a calculation in $\Pro^{2g+2}$.
\begin{lemma} \label{lem.reduction}
Let $\xi$ denote the first Chern class of $\Oc_{\Pro^{2g+2}}(1)$. Then
the pull-back
$$
A^*_{\GL_2}(\Pro^{2g+2} \backslash \Delta_1) \xrightarrow{\Pi^*} A^*_{\GL_2}(\Sym^{2g+2}E\otimes
{\mathcal D}^{\otimes g} \backslash \Delta_1)
$$
is surjective and its kernel is generated by $gc_1-\xi$.
\end{lemma}
\begin{dem}
In general if $\Pi: X \to Y$ is a $G$-equivariant $\Gm$-bundle we consider the associated line bundle $\mathcal L$. Then $X$ is the complement of the 0-section and the localization exact sequence for equivariant Chow groups implies that
$$
A^*_G(Y)/c_1(\mathcal L) \cong A^*_G(X)
$$ 
where the isomorphism is induced by $\Pi^*$.  In our case the
associated line bundle is ${\mathcal D}^{\otimes g} \otimes \mathcal O
(-1)$ so the kernel of $\Pi^*$ is $gc_1- t$.
\end{dem}

From Lemma \ref{lem.reduction} we conclude that
\begin{equation}\label{quot}
A^*{\Hcal_g}=A^*_{\GL_2}(\Pro^{2g+2} \backslash \Delta_1)/(gc_1-\xi)
\end{equation}
so we have reduced the problem to computing
\begin{equation*}
A^*_{\GL_2}(\Pro^{2g+2} \backslash \Delta_1) = A^*_{\GL_2}(\Pro^{2g+2})/I
\end{equation*}
where $I$ is the ideal generated by the image of $A^{\GL_2}_*(\Delta_1)$.

\section{The equivariant Chow ring of the space of non-degenerate
homogeneous forms in $N$ variables}
Let $E$ be the defining representation of $\GL_2$
and let $\Pro^N = \Pro(\Sym^N E^*)$ be the 
projective space of homogeneous forms of degree $N$ in two variables $x_0$ and $x_1$. Since 
the action of $\GL_2$ commutes with homotheties,
there is an induced action of $\GL_2$ on $\Pro^N$
with kernel the center of $\GL_2$.

Let $\Delta_1 \subset \Pro^N$ be the locus of forms
which are divisible by a square over some
extension of the base field. This subvariety is $\GL_2$ invariant
and the goal of this section is to compute the 
equivariant Chow ring $A^*_{\GL_2}(\Pro^N - \Delta_1) =
A^*_{\GL_2}(\Pro^N)/I$ where $I$ is the
ideal in $A^*_{\GL_2}(\Pro^N)$ generated by the image of the equivariant
Chow groups of $\Delta_1$. We can then apply the results 
of this section when $N = 2g+2$ to complete the computation 
of the integral Chow ring of ${\mathcal H}_{g}$ for $g$ even.
(Note, however, that in this section we do not require $N$ to be even.)

For every $r=1, \dots, [N/2]$ we
may define $\Delta_r \subset \Sym^{N} E^*$ as the closed subvariety of forms
divisible by the square of a polynomial of degree $r$ over
some extension of the ground field $k$. The locus $\Delta_r$
is the image of the map
\begin{eqnarray*}
\pi_r: \Sym^r E^* \times \Sym^{N-2r}E^* &\to& \Sym^{N}E^*\\
(f,g) &\mapsto& f^2g
\end{eqnarray*}
Passing to the associated projective spaces we obtain $\GL_2$-equivariant maps
\begin{equation}\label{eq.pir}
\pi_r: \Pro^r \times \Pro^{N-2r} \to \Pro^{N}
\end{equation}
and $\Delta_r$ will still indicate the images of $\pi_r$

\begin{prop} (cf. \cite[ Lemma 3.3]{Vis:98}) \label{prop.stratpn}
If the characteristic of $k> (N-2)$ 
then the image of $A_*^{\GL_2}(\Delta_1)$ in $A^*_{\GL_2}(\Pro^N)$
is the sum of the images of the homomorphisms
$$\pi_{r*}\colon A_*^{\GL_2}(\Pro^r \times \Pro^{N-2r}) \to A^*_{\GL_2}(\Pro^N)$$
\end{prop}
\begin{proof}
Let $G$ be an algebraic group. Recall \cite{EdGr:98} that if $X$ is a $G$-scheme then an {\em equivariant envelope}
is a proper $G$-equivariant morphism $f \colon \tilde{X} \to X$
such that for every $G$-invariant subvariety $W \subset X$ there
is a $G$-invariant subvariety $\tilde{W} \subset \tilde{X}$ mapping
birationally to $X$. Lemma 3 of \cite{EdGr:98} together with \cite[Lemma 18.3(6)]{Ful:84}
implies that proper pushforward $f_* \colon A_*^G(\tilde{X}) \to A_*^G(X)$ is
surjective.

Let $$\pi \colon \coprod_{r=1}^{N/2} \Pro^r \times \Pro^{N-2r} \to
\Delta_1$$ be the $\GL_2$-equivariant map whose restriction to $\Pro^r
\times \Pro^{N-2r}$ is $\pi_r$. By the above discussion it suffices to prove
that $\pi$ is an equivariant envelope. The argument used in the proof
\cite[Lemma 3.2]{Vis:98} shows that if $\Char k > N-2r$, then for every
field $K \supset k$, if $p$ is $K$-valued point of $\Delta_r
\smallsetminus \Delta_{r+1}$ then there is a (unique) $K$-valued point
$q$ of $\Pro^r \times \Pro^{N-2r}$ mapping to $p$.

Let $Z \subset \Delta_{1}$ be a $G$-invaraint subvariety and suppose
that the generic point $p$ of $Z$ lies in $\Delta_{r} \smallsetminus
\Delta_{r+1}$. By Vistoli's Lemma we know that there is subvariety
$\tilde{Z} \subset \Pro^r \times \Pro^{N-2r}$ mapping birationally to
$Z$. To complete the proof we must show that we may take $\tilde{Z}$
to be $G$-invariant. Now if $g \in \GL_2$ then the subvariety
$g\tilde{Z}$ also maps birationally to $V$. Since there is a unique
point of $\Pro^r \times \Pro^{N-2r}$ mapping to the generic point $q$,
it follows that $g\tilde{Z}$ and $\tilde{Z}$ have the same generic
point. Hence $\tilde{Z}$ must contain a $G$-invariant open set
$\tilde{U}$. Taking the closure of $\tilde{U}$ in $\Pro^r \times
\Pro^{N-2r}$ gives our desired $G$-invariant subvariety mapping
birationally to $Z$.
\end{proof}

For any $r$ let $\xi_{r,1}$ (respectively $\xi_{N-2r,2}$) be the pullback to
$\Pro^r \times \Pro^{N-2r}$ of $c^{\GL_2}_1(\Oc(1))$ on the first
(resp. second) factor. We have the following relation in
$A^*_{\GL_2}(\Pro^r \times \Pro^{N-2r})$:
$$
\pi^*_r(\xi)=2 \xi_{r,1} + \xi_{N-2r,2}.
$$
Moreover, from Lemma \ref{lem.projspace} we have that $\xi_{1,r}$ is a
zero of a monic polynomial of degree $r+1$ with coefficients in
$A^*_{\GL_2}$. Therefore $\pi_{r*}(A^*_{\GL_2}(\Pro^r \times
\Pro^{N-2r}))$ is generated as a $A^*_{\GL_2}(\Pro^{N})$-module
by $\pi_{r*}(1), \pi_{r*}(\xi_{1,r}), \dots, \pi_{r*}(\xi^r_{1,r})$.

For $r=1, \dots, [N/2]$ and $i=0, \dots, r$ set
$$
\alpha_{r,i}:= \pi_{r*}(\xi^i_{r,1}).
$$
The above discussion allows us to conclude that
\begin{equation*}
A^*_{\GL_2}(\Pro^{N} \backslash \Delta_1) = A^*_{\GL_2}(\Pro^{N})/(\{\alpha_{r,i}\})
\end{equation*}

\subsection{The ideal generated by $A^*_T(\Delta_1)$ in
$A^*_T(\Pro^N)$}
The action of $\GL_2$ on $\Pro^N$ restricts to 
an action of the maximal torus $T \subset \GL_2$
consisting of diagonal matrices. The goal of this
section is to prove 

\begin{prop} \label{prop.deg}
The image of $A_*^T(\Delta_1)$ in $A^*_T(\Pro^N)$ is
the ideal $(\alpha_{1,0}, \alpha_{1,1})$.
\end{prop}
(Here we use the notation $\alpha_{r,i}$ to indicate the
restriction of the same named classes to $A^*_T(\Pro^N)$.)

The proof will require the introduction of alternate
but, less symmetric classes which generate the same ideal.
The group $T$ acts on $\Pro^1$ by
$z\cdot (a\colon b) = (\lambda_1^{-1}(z) a \colon \lambda_2^{-1}(z) b)$.  
Choose coordinates
$(X_0: X_1 \colon \ldots \colon X_N)$ on $\Pro^N$ so that the
coordinate function $X_i$ is the coefficient of $x_0^{N-i}x^i_1$ in
a homogeneous form of degree $N$. Then $T$ acts on $\Pro^N$ by the rule
\begin{equation*}
z\cdot (X_0  \colon X_1 \colon \ldots
\colon X_{N-1} \colon X_N) = 
(\lambda_1^{-N}(z) X_0 \colon \lambda_1^{1-N}\lambda_2^{-1}(z) \colon
\ldots \colon \lambda_1^{-1}  \lambda_2^{1-N}(z) X_{N-1} \colon 
\lambda_2^{-N}(z)X_N).
\end{equation*}

Let $H_i \subset \Pro^N$  be the hyperplane defined by
the equation $X_i = 0$; in other words, $H_i$ corresponds
to forms $f$ such that the coefficient of $x_0^{N-i}x_1^i$ is
$0$. Let $h_i$ be the $T$-equivariant fundamental class of $H_i$.

By Lemma \ref{lem.Thypersurface}, we have that 
\begin{equation} \label{eq.hi}
\begin{array}{lcl}
h_i & = & c_1({\mathcal O}(1)) + c_1(\lambda_1^{i-N}\lambda_2^{-i})\\
& = & \xi - (N-i)t_1 - it_2
\end{array}
\end{equation}
On $\Pro^r$ we can consider the hyperplane class $h_{i,r}$
corresponding to degree $r$ forms such that the coefficient
of $x_0^{r-i}x^i_1$ is 0. 
Again by Lemma \ref{lem.Thypersurface} we have
that 
$h_{i,r} = \xi_{1,r} - (r-i)t_1 - it_2$.
It follows that 
the image of $A_*^T(\Delta_1)$ is generated by 
the pushforwards to $A^*_T(\Pro^N)$ of the classes $\pi_{r*}(h_{0,r} \ldots h_{m,r})$
for $1 \leq r \leq [N/2]$, $0 \leq m \leq r-1$
as well as the classes
$\alpha_{r,0}$.

\begin{lemma} \label{lem.alpha10alpha11}
In $A^*_T(\Pro^N)$ we have the relations
\begin{equation} \label{eq.alpha10}
\begin{array}{lcl}
\alpha_{1,0} & = & 2(N-1)h_0 + N(N-1)(t_1 -t_2) \\ 
& = & 2(N-1)\xi  - N(N-1)c_1
\end{array}
\end{equation}
\begin{equation} \label{eq.alpha11}
\begin{array}{lcl}
\alpha_{1,1} & = & h_0h_1 + t_1 \alpha_{1,0} \\ 
& = &  \xi^2 - c_1 \xi  - N(N-2) c_2 
\end{array}
\end{equation}
\end{lemma}
(Recall that $c_1 = t_1 + t_2$ and $c_2 = t_1 t_2$ are the 
elementary symmetric polynomials in the generators $t_1,t_2$ for
$A^*_T$.)
\begin{remark}
Note that the identities $\alpha_{1,0} = 2(N-1)\xi  - N(N-1)c_1$
and $\alpha_{1,1} = \xi^2 - c_1 \xi  - N(N-2) c_2 $ also hold
in $A^*_{\GL_2}(\Pro^N)$ since the restriction map $A^*_{\GL_2}(\Pro^N)
\to A^*_{T}(\Pro^N)$ is injective.
\end{remark}
\begin{proof}
The multiplication map $\rho \colon (\Pro^1)^N \to  \Pro^N$ is $\GL_2$-equivariant and hence $T$-equivariant. It
also commutes with the natural permutation action of $S_N$
on $(\Pro^1)^N.$
and with this notation the subvariety  $\Delta_1$ 
corresponding to homogeneous forms with multiple roots
is the image of $\Delta_{\Pro^1 \times \Pro^1} \times (\Pro^1)^{N-2}$.

Consider the diagram
\begin{equation} \label{diag.delta1} \begin{array}{ccc} 
({\Delta}_{\Pro^1 \times \Pro^1}) \times (\Pro^1)^{N-2} 
&  \hookrightarrow & (\Pro^1)^N\\
\downarrow \rho_1 & & \downarrow \rho\\
\Delta_r & \stackrel{i_1} \hookrightarrow & \Pro^N
\end{array}
\end{equation}
Where $\rho$ is the multiplication map and $\rho_1$
is its restriction to $\Delta_{\Pro^1 \times \Pro^1} \times (\Pro^1)^{N-2}$.
Since $\rho_1$ has degree $(N-2)! $we see that
\begin{equation} \label{eq.alpha10.3}
(N-2)! \alpha_{1,0} = \rho_*[\Delta_{\Pro^1 \times \Pro^1} \times (\Pro^1)^{N-2}]
\end{equation}

The torus $T$ acts on $\Pro^1$ by $t(x_0 \colon x_1)
= (\lambda_1^{-1} x_0 \colon \lambda_2^{-1}x_1)$ so by Lemma \ref{lem.p1diag} 
\begin{equation} \label{eq.p1diag}
[\Delta_{\Pro^1 \times \Pro^1}] = [0 \times \Pro^1] \times [0 \times \Pro^1]
+ (t_1 -t_2)[\Pro^1 \times \Pro^1].
\end{equation}
Substituting the right hand side of \eqref{eq.p1diag}
into the right hand side of \eqref{eq.alpha10.3}
we obtain
\begin{equation} \label{eq.alpha10.4}
(N-2)!\alpha_{1,0} = \rho_*\left([0 \times \Pro^1 \times (\Pro^1)^{N-1}]
+ [\Pro^1 \times 0 \times (\Pro^1)^{N-1}] + (t_1 -t_2)[(\Pro^1)^N]\right)
\end{equation}
The first two terms on the right hand side of \eqref{eq.alpha10.4}
pushforward to $(N-1)! h_0$ since the map $0 \times (\Pro^1)^{N-1}
\to H_0$ has degree $(N-1)!$ while the direct image of the second term is
$N! (t_1 -t_2)$.
Thus we obtain the relation
\begin{equation} \label{eq.alpha10.5}
(N-2)!\alpha_{1,0} = 2(N-1)!h_0 + N!(t_1 -t_2)
\end{equation}
Since $A^*_T(\Pro^N)$ is torsion free we can divide
\eqref{eq.alpha10.5} by $(N-2)!$ to obtain the first identity in
\eqref{eq.alpha10} and 
substituting $h_0 = \xi - N t_1$ yields the second.

As noted above we have  $h_{0,1} = \xi_{1,1} - t_1$ in $A^*_T(\Pro^1)$.
Thus $\alpha_{1,1} = \pi_{1*}h_{0,1} + t_1[\Delta_1]$.
Now $(N-2)! h_{0,1} = \rho_{1*}([0 \times 0 \times \Pro^{N-2}])$. Thus,
after pushing forward to $\Pro^N$ we obtain the identity
\begin{eqnarray*}
(N-2)! i_{1*}\pi_{1*} h_{0,1} & = & (N-2)! 
\rho_*[0 \times 0 \times (\Pro^1)^{N-2}]\\
& = & (N-2)!h_0h_1
\end{eqnarray*}
Since $i_{1*} [\Delta_1] = \alpha_{1,0}$ the first identity
in \eqref{eq.alpha11} follows. Substituting $h_0 = \xi  -Nt_1$
and $h_1 = \xi + (1-N)t_1 - t_2$ yields the second.
\end{proof}

As an immediate consequence of the identities in Lemma \ref{lem.alpha10alpha11} we obtain.
\begin{lemma}
$(\alpha_{1,0} , \alpha_{1,1}) = (\alpha_{1,0}, h_0h_1)$ as
ideals in $A^*_T(\Pro^N)$.
\end{lemma}

If $m \geq 0$ let $\beta_{r,m}$ be the image of the class $h_{r,0} \ldots h_{r,m}$
in $A^*_T(\Pro^m)$.
\begin{lemma} \label{lem.hrm}
The class $\beta_{r,m}$ is 
a multiple of $h_0 h_1$.
\end{lemma}

\begin{proof}
Consider the diagram analogous to \eqref{diag.delta1}
\begin{equation} \label{diag.deltar} \begin{array}{ccc} 
({\Delta}_{\Pro^1 \times \Pro^1})^r \times (\Pro^1)^{N-2r} 
&  \hookrightarrow & (\Pro^1)^N\\
\downarrow \rho_r & & \downarrow \rho\\
\Delta_r & \stackrel{i_r} \hookrightarrow & \Pro^N
\end{array}
\end{equation}

Identifying $\Delta_{\Pro^1 \times \Pro^1}$ with $\Pro^1$
then $h_{r,0} \cdot \ldots \cdot h_{r,m}$ is the equivariant
fundamental class of the image of the $T$-invariant
subvariety 
$$(0 \times 0) \times  \ldots (0 \times 0) \times
(\Delta_{\Pro^1 \times \Pro^1})^{r-(m+1)} \times (\Pro^1)^{N-2r}.$$
(Here the first $(2m +2)$ coordinates are 0).
Let $\theta_{r,m}$ be the equivariant fundamental class of 
this $T$-invariant subvariety
Thus,  $\rho_{r*}\theta_{r,m} = (N-2r)!(r-m-1)! h_{r,0} \cdot \ldots \cdot h_{r,m}$.
On the other hand we may expand $\theta_{r,m}$ by replacing 
$\Delta_{\Pro^1 \times \Pro^1}$ with the formula of \eqref{eq.p1diag}
to
obtain a sum
of classes which are permutations of classes of the form
\begin{equation} \label{eq.zeros}
(t_1 - t_2)^{m+r+1 -l}[0 \times 0 \ldots \times 0 \times (\Pro^1)^{N-l}]
\end{equation}
where first $l$ coordinates are 0 and $2m +2 \leq l \leq 
r + m +1$.
The action of $T$ on $(\Pro^1)^N$ commutes with natural
permutation action of $S_N$ and the map $\rho$ is $S_N$-equivariant
the pushforward $\rho_*$ of the classes in \eqref{eq.zeros}
we obtain the identity
\begin{equation}
(r-(m+1))!(N-2r)! \beta_{r,m} = \sum_{l = 2m +2}^{m + r + 1} a_{l} (N-l)!(t_1 -t_2)^{m+r+1-l}
h_0 \cdot \ldots \cdot h_l
\end{equation}
where the $a_l$'s are positive integers.
Since $l \leq  m+r +1$ it follows that
$(r- (m+1))!(N-2r)!$ divides $(N-l)!$ \footnote{
This follows because $a!(N-b)!$ always
divides $(N-b+a)!$ since $\frac{(N-b+a)!}{a!(N-b)!}$ is
a binomial coefficient, so $a!(N-b)!$ divides $(N-k)!$
if $k \leq b-a$.} Thus, since $A^*_T(\Pro^N)$ is torsion free
$\alpha_{r,m}$ is an integral sum of terms of the form 
$(t_1-t_2)^{m+r+1-l} h_0 \cdot \ldots \cdot h_l$
which is clearly a multiple of $h_0h_1$.
\end{proof}

\begin{lemma} \label{lem.alpha20}
$\alpha_{2,0}$ is in the ideal $(\alpha_{1,0}, h_0h_1)$.
\end{lemma}
\begin{proof}
Since $\Delta_{2}$ is the image of the fundamental class
of $\Delta_{\Pro^1 \times \Pro^1} \times \Delta_{\Pro^1 \times \Pro^1}
\times (\Pro^1)^{N-4}$ we see that
\begin{equation} \label{eq.delta2}
2!(N-4)! \alpha_{2,0} = \\
\rho_*[\Delta_{\Pro^1 \times \Pro^1} \times \Delta_{\Pro^1 \times \Pro^1} \times (\Pro^1)^{N-4}]
\end{equation}
Expand the first diagonal term as 
$$[\Pro^1 \times 0] + [0 \times \Pro^1]
+ (t_1 -t_2)[\Pro^1 \times \Pro^1]$$ 
and substitute this into the right-hand-side of equation \eqref{eq.delta2}
to obtain
\begin{multline} \label{eq.delta2.1}
2!(N-4)!\alpha_{2,0} =
\rho_*[0 \times \Pro^1 \times \Delta_{\Pro^1 \times \Pro^1} 
\times (\Pro^{1})^{N-4}] \\
+ \rho_*[\Pro^1 \times 0 \times \Delta_{\Pro^1 \times \Pro^1} 
\times (\Pro^{1})^{N-4}]
+(N-2)!(t_1 - t_2)\alpha_{1,0}
\end{multline}
Since the action of $T$ on $(\Pro^1)^N$ commutes with the permutation
action of $S_N$ the first two terms in the right-hand side of
\eqref{eq.delta2.1} are equal.
Now expand the remaining $\Delta_{\Pro^1 \times \Pro^1}$ as above 
so that the sum of the first
two terms on the right hand side of \eqref{eq.delta2.1} now becomes
\begin{multline} \label{eq.delta2.2}
2\rho_*[\Pro^1 \times 0 \times \Pro^1 \times 0 \times (\Pro^1)^{N-4}]
+\rho_*[\Pro^1 \times 0 \times 0 \times \Pro^1 \times (\Pro^1)^{N-4}] \\
+(t_1 -t_2)\rho_* [\Pro^1 \times 0  \times \Pro^1 \times \Pro^1 \times (\Pro^1)^{N-4}]
\end{multline}
Using the fact $\rho_*$ commutes with the permutation action
of $S_N$ we can combine the terms coming from \eqref{eq.delta2.2}
with the last term on the right side of equation \eqref{eq.delta2.1}
to obtain
\begin{equation} \label{eq.delta2.3}
2!(N-4)! \alpha_{2,0} = 4(N-2)! h_0 h_1\\ + 2(N-1)!(t_1 - t_2)h_0
+(N-2)!(t_1 - t_2) \alpha_{1,0}
\end{equation}
Dividing through by $2!(N-4)!$ and again invoking the fact
that $A^*_T(\Pro^N)$ is torsion free we obtain
the equation
\begin{multline} \label{eq.delta2.4}
\alpha_{2,0} = 2(N-1)(N-2)(N-3) h_0h_1 + (N-1)(N-2)(N-3)(t_1 -t_2)h_0\\
+ \frac{(N-2)(N-3)}{2} (t_1 -t_2)\alpha_{1,0}
\end{multline}
The first and third terms on the right-hand-side are clearly in
the ideal $(\alpha_{1,0},h_0h_1)$. The middle term
is in the ideal because by \eqref{eq.hi} and \eqref{eq.alpha10}
\begin{equation} \label{eq.th0}
(N-1)(N-2)h_0(t_1 -t_2) = h_0\alpha_{1,0} - 2(N-1)h_0h_1
\end{equation}
\end{proof}
Our last lemma completes the proof of the proposition.
\begin{lemma} \label{lem.alphar0}
If $r\geq 3$ then $\alpha_{r,0}$ is in the ideal
$(\alpha_{1,0}, h_0h_1)$.
\end{lemma}
\begin{proof}
As above we have
\begin{equation} \label{eq.delta_r}
r!(N-2r)! \alpha_{r,0} 
= \rho_* [(\Delta_{\Pro^1 \times \Pro^1})^r \times (\Pro^1)^{N-2r}]
\end{equation}
Expanding out the first $r-2$ copies of 
the fundamental class of $\Delta_{\Pro^1 \times \Pro^1}$ 
in the right-hand side of \eqref{eq.delta_r} we obtain
that  $(\Delta_{\Pro^1 \times \Pro^1})^r \times (\Pro^{1})^{N-2r}$ is the sum
of 
$$(t_1 -t_2)^{r-2} 
[(\Delta_{\Pro^1 \times \Pro^1})^2 \times (\Pro^{1})^{N-4}]$$
plus terms which are permutations of the $r-2$ classes
$$(t_1 - t_2)^{r-2-k} [(0 \times \Pro^1)^{k}
\times (\Delta_{\Pro^1 \times \Pro^1})^2]
\times (\Pro^1)^{N-2(k+2)}$$ where  $1 \leq k \leq r-2$.
Since $$\rho_*[(\Delta_{\Pro^1 \times \Pro^1})^2 \times (\Pro^{1})^{N-4}]= 2(N-4)!\alpha_{2,0}$$
there are positive integers $b_1, \ldots b_{r-2}$ such that
\begin{multline} \label{eq.deltar}
r!(N-2r)! \alpha_{r,0} = 2(N-4)!(t_1 - t_2)^{r-2} \alpha_{2,0}\\
+ \sum_{k=1}^{r-2} b_k (t_1 -t_2)^{r-2-k}\rho_*[(0 \times \Pro^1)^{k} 
\times (\Delta_{\Pro^1 \times \Pro^1})^2 
\times (\Pro^1)^{N-2(k+2)}]
\end{multline}
Now expand 
$[(\Delta_{\Pro^1 \times \Pro^1})^2]$ in $A^*_T((\Pro^1)^4)$
as 
\begin{equation} \label{eq.deltasq} 
([\Pro^1 \times 0] + [0 \times \Pro^1] +(t_1 -t_2) \Pro^1
\times \Pro^1)\times ([\Pro^1 \times 0] + [0 \times \Pro^1]
+(t_1 -t_2) [\Pro^1 \times \Pro^1)].
\end{equation}
When  $k \geq 2$ we substitute \eqref{eq.deltasq} and use
the permutation invariance of $\rho_*$ to obtain 
\begin{multline} \label{eq.deltarr}
\rho_*[(0 \times \Pro^1)^{k} \times (\Delta_{\Pro^1 \times \Pro^1})^2
\times (\Pro^1)^{N-2(k+2)}] =  (N-k)!(t_1  -t_2)^2 h_0 \ldots h_{k-1}\\
+ 4(N-(k+2))!(t_1 -t_2)h_0\ldots h_{k+1} + 4(N - (k+4))! h_0 \ldots
h_{k+3}.
\end{multline}
When $k =1$ we use the same expansion to obtain
that 
\begin{multline} \label{eq.deltarrr}
\rho_*([0 \times \Pro^1 \times (\Delta_{\Pro^1 \times \Pro^1})^2 \times
(\Pro^1)^{N-6}]) = 4(N-3)! h_0 h_1 h_2 \\+ 4(N-2)!(t_1 -t_2)h_0h_1 +
(N-1)!(t_1 -t_2)^2h_0
\end{multline}
Since $r \geq 3$, the first term on the right-hand-side of \eqref{eq.deltar}
is  divisible $r!(N-2r)! \alpha_{2,0}$ and thus is an element of $r!(N-2r)!(\alpha_{1,0},
h_0h_1)$
by Lemma \ref{lem.alpha20}.
Likewise, each of the three terms on the right-hand-side 
of  \eqref{eq.deltarr} is divisible $r!(N-2r)! h_0 h_1$
as is the
first  term on the right-hand-side of \eqref{eq.deltarrr}.
Finally by \eqref{eq.th0} we know that $(N-1)(N-2)(t_1 -t_2)h_0 \in (\alpha_{1,0}, h_0h_1)$.
Hence $(N-1)!(t_1 -t_2)h_0$ is in $(N-3)! (\alpha_{1,0}, h_0h_1)$. Since $r \geq 3$,
$r!(N-2r)! | (N-3)!$, so $(N-1)!(t_1 - t_2) h_0$ is in $r!(N-2r)! (\alpha_{1,0},h_0h_1)$.
Hence all of the terms appearing on the right hand side of
\eqref{eq.deltarrr} are in $r!(N-2r)!(\alpha_{1,0}, h_0h_1)$. Since $A^*_T(\Pro^N)$ is torsion
free we conclude that $\alpha_{r,0} \in (\alpha_{1,0}, h_0h_1)$.
\end{proof}

\subsection{The ideal generated by the image of $A^{\GL_2}_*(\Delta_1)$
in $A^*_{\GL_2}(\Pro^N)$} \label{sec.gl2}
We are now in a position to prove our main result
\begin{thm} \label{thm.deg}
The image of $A_*^{\GL_2}(\Delta_1)$ in $A^*_{\GL_2}(\Pro^N)$ is the ideal $I  = (\alpha_{1,0},
\alpha_{1,1})$.
\end{thm}
\begin{remark}
When $N = 6$ this result was previously obtained by Vistoli 
\cite{Vis:98} for his calculation of $A^*{\mathcal M}_2$.
\end{remark}
\begin{proof}
Since $\GL_2$ is special, $A^*_{\GL_2}(\Pro^N)$ is (non-canonically) a
summand in the $A^*_{\GL_2}(\Pro^N)$-module $A^*_T(\Pro^N)$ by Proposition \ref{prop.tsummand}.
Thus we may decompose the $A^*_{\GL_2}(\Pro^N)$-module $A^*_T(\Pro^N)$
as $A^*_T(\Pro^N) = A^*_{\GL_2}(\Pro^N) \oplus M$
where $M \subset A^*_{T}(\Pro^N)$ is a complimentary submodule.

Now suppose that $f \in A^*_{\GL_2}(\Pro^N)$ is in the
image of $A_*^{\GL_2}(\Delta_1)$. Since the inclusion
of $A^*_{\GL_2}(\Delta_1)$ in $A^*_T(\Delta_1)$
commutes with the direct image map $i_* \colon \Delta_1 \to \Pro^N$,
we may view $f$ as being in the image of $A^*_T(\Delta_1)$.
Thus by our previous proposition we may write
$f = a \alpha_{1,0} + b \alpha_{1,1}$ for some $a,b \in A^*_T(\Pro^N)$.
Using the decomposition above we may write
$a = a_s + a_u$, $b = b_s + b_u$  with $a_s, b_s$ 
in $A^*_{\GL_2}(\Pro^N)$ 
and $a_u, b_u$ in $M$. Thus
$$f = a_s \alpha_{1,0} + b_s \alpha_{1,1} + a_u \alpha_{1,0} + b_u \alpha_{1,1}.$$
Since $f \in A^*_{\GL_2}(\Pro^N)$ it follows that
$a_u \alpha_{1,0} + b_u \alpha_{1,1} $ is an element of 
$M \cap A^*_{\GL_2}(\Pro^N)=\{0\}$.
Hence $f = a_s \alpha_{1,0} + b_s \alpha_{1,1}$, so 
$f$ is in the ideal of $A^*_{\GL_2}(\Pro^N)$ generated by
$\alpha_{1,0}$ and $\alpha_{1,1}$.
\end{proof}

\section{Proof of Theorem \ref{thm.main}} 
We can now easily complete the proof of Theorem \ref{thm.main}
By Theorem \ref{thm.deg} 
$$A^*_{\GL_2}(\Pro^{2g+2} - \Delta_1) =
A^*_{\GL_2}(\Pro^{2g+2})/(\alpha_{1,0},\alpha_{1,1})$$ where
\begin{eqnarray}
\alpha_{1,0} & = & 2(2g+1)\xi -(2g+2)(2g+1)c_1 \label{eq.alpha10xi} \\
\alpha_{1,1} & = & \xi^2 - \xi_1 c_1 - (2g+2)(2g)c_2 \label{eq.alpha11xi}
\end{eqnarray}
If $C_1, \ldots , C_{2g+3}$ are the $\GL_2$-equivariant Chern
classes of the representation $\Sym^{2g+2}E^*$
set 
\begin{equation} \label{eq.pxi}
P = \xi^{2g+3} + C_1 \xi^{2g+1} + \ldots C_{2g+2}
\end{equation}
so that $A^*_{\GL_2}(\Pro^{2g+2}) = \Z[c_1,c_2][\xi]/P$ by Lemma \ref{lem.projspace}.
Let $\alpha_{1,0}(gc_1)$, $\alpha_{1,0}(gc_1)$, $P(gc_1)$
be the polynomials in $c_1, c_2$ obtained by substituting
$\xi =g c_1$ in \eqref{eq.alpha10xi}-\eqref{eq.pxi} then
by Lemma \ref{lem.reduction},
$$A^*({\mathcal H}_g)= \Z[c_1,c_2]/(\alpha_{1,0}(gc_1), \alpha_{1,1}(gc_1), P(gc_1))$$
where 
\begin{eqnarray*}
\alpha_{1,0}(gc_1) & =& -2(2g+1)c_1\\
\alpha_{1,1}(gc_1) & = & -g(g-1)c_1^2 + 4g(g+1)c_2
\end{eqnarray*}

The theorem then follows from our last lemma.
\begin{lemma} \label{lem.pgc1}
The polynomial $P(gc_1)$ is in the ideal 
$\left(\alpha_{1,0}(gc_1), \alpha_{1,1}(gc_1)\right)$.
\end{lemma}
\begin{proof}
Using the same arguments as in Section \ref{sec.gl2} 
it suffices to show that the restriction of $P(gc_1)$ to $A^*_T$
is in the ideal generated by the restrictions of $\alpha_{1,0}(gc_1)$
and $\alpha_{1,1}(gc_1)$.

The Chern roots of $E$ are $\{-(2g+2) t_1, -(2g+1)t_1 -t_2, \ldots , 
-t_1 -(2g+1)t_2, -(2g+2)t_2\}$
so 
\begin{equation}\label{eq.pc1.1}
P(gc_1)  = \prod_{i=0}^{2g+2} (gc_1 - (2g+2 -i)t_1 - it_2)
\end{equation}
Pairing off the $2g$ terms $(gc_1 - (2g+2 -i)t_1 - it_2)$
and $(gc_1 -it_1 - (2g+2-i)t_2)$ and observing
that if $i = g+1$ then $gc_1 - (2g+2 - i)t_1 - it_2
= -c_1$ we can rewrite \eqref{eq.pc1.1}
as
\begin{equation} \label{eq.pc1.2}
P(gc_1) = -c_1\prod_{j=0}^{g}\left( -(g-j)(g-j+2)c_1^2 +4(g+1 -j)^2c_2\right)
\end{equation}
The product of the $j = 0$ and $j = 1$ terms on the right hand side of
\eqref{eq.pc1.2} is
\begin{eqnarray*}
Q(c_1,c_2) & = & ((g-1)(g+1)c^2_1 - 4 g^2c_2)(g(g+2)c^2_1 - 4(g+1)^2c_2)\\
& = & 
(\alpha_{1,0}(g c_1))^2 c_2 + \alpha_{1,0}(g c_1) \alpha_{1,1}(g c_1)c_1 + (\alpha_{1,1}(g c_1))^2 
\end{eqnarray*}
\end{proof}

\subsection{Tautological classes} \label{sec.taut}
The identification, for $g$ even,  
of ${\mathcal H}_g = [({\mathbb A}^{2g+3} \smallsetminus 
\Delta_1)/\GL_2]$ means that the defining representation of
$\GL_2$ determines a vector bundle on ${\mathcal H}_g$
whose Chern classes generate the Chow ring.
Using an observation of Gorchinskiy and Viviani 
we can obtain 
a functorial geometric description of this bundle.

If $\pi \colon X \to S$ is a family of smooth hyperelliptic curves
of genus $g$
let $W \subset X$ be the divisor of Weierstrass points of the 
fibers of $\pi$ and let $\omega$ be the relative canonical
line bundle. Then ${\mathcal O}_{X}(W)$ has relative degree
$2g+2$ and the line bundle  
$\omega_{\pi}^{\otimes g/2} \otimes {\mathcal O}_X((1 -g/2)W)$
restricts to a $g^1_2$ on the fibers of $\pi$. Since $\pi$ is flat,
it follows that 
$$V_\pi = \pi_*(\omega_{\pi}^{\otimes g/2} \otimes {\mathcal O}_X((1 -g/2)W)$$
is rank 2 vector bundle on $S$. Let 
${\mathcal V}_g$ be the rank two bundle on the stack ${\mathcal H}_g$ which restricts to 
to $V_{\pi}$ on a family of hyperelliptic curves $X \stackrel{\pi} \to S$.
\begin{prop} \label{prop.taut}
Under the identification ${\mathcal H}_g = [({\mathbb A}^{2g+3}\smallsetminus
\Delta_1)/\GL_2]$ the vector bundle ${\mathcal V}_g$ corresponds
to the defining representation of $\GL_2$. In particular the
Chow ring of ${\mathcal H}_g$ is generated by the Chern classes
of ${\mathcal V}_g$.
\end{prop}
\begin{remark}
When $g = 2$, then ${\mathcal V}_g$ is the Hodge bundle and we recover
Vistoli's result \cite{Vis:98} that $A^*({\mathcal M}_2)$ is generated
by the Chern classes of the Hodge bundle.
\end{remark}
\begin{remark}
As observed in \cite{GoVi:06} the Chern classes of ${\mathcal V}_g$ 
are not tautological
classes. Gorchinskiy and Viviani show that the first Hodge class
$\lambda$ equals $(g/2) c_1$. In principal the methods of their paper 
could be extended to give formulas for all of the tautological
classes in terms of $c_1$ and $c_2$, but we do not pursue this here.
\end{remark}
\begin{proof}
Following \cite{ArVi:04} and \cite{GoVi:06}
we know that given a family of hyperelliptic curves $\pi \colon X \to S$
the map $\pi$ factors as $\pi = p \circ f$ where $p \colon P \to S$
is a Brauer-Severi variety and $f \colon X \to P$ is a double 
cover. Then $f_*{\mathcal O_X}= O_P \oplus {\mathcal L}$ where
${\mathcal L}$ is a line-bundle such that ${\mathcal L}^2 = {\mathcal O_P}(-D)$
where $D \subset P$ is the ramification divisor. Since $f$ is
a double cover, the family of elliptic curves is uniquely determined
by the data $(p \colon P \to   S, {\mathcal L})$.

The identity $f^*{\mathcal L}^{-1} = {\mathcal O}_X(W)$ 
and the Riemann-Hurwitz formula 
imply that 
$$f_*({\omega}^{\otimes g/2} \otimes {\mathcal O}_X((1-g/2)W) = 
{\omega_{P/S}}^{\otimes g/2} \otimes {\mathcal L}^{-1} \otimes ({\mathcal O}_P \oplus
{\mathcal L})$$
Since the restriction of $\omega_{P/S}^{\otimes g/2}$ has degree $-g$
we see that 
$$\pi_*(\omega^{\otimes g/2} \otimes {\mathcal O}_X((1-g/2)W) = 
p_*(\omega_{P/S}^{\otimes g/2} \otimes {\mathcal L}^{-1}).$$

As noted in \cite[Remark 3.3]{ArVi:04} we
may identify ${\mathcal H}_g$ with the stack ${\mathcal H}'_{1,2,g+1}$
whose objects over a $k$-scheme $S$ consists of the data
of a Brauer-Severi variety $p \colon P \to S$ together with
a line bundle ${\mathcal L}$ of relative degree $-(g+1)$
and an injection  $i \colon {\mathcal L}^2 \hookrightarrow {\mathcal O}_P$.
From the previous paragraph we see that we may identify ${\mathcal V}_g$
with the bundle whose restriction to $P \to S$ is the
vector bundle $p_*(\omega_{P/S}^{\otimes g/2} \otimes {\mathcal L}^{-1})$.

Let $\tilde{{\mathcal H}}_{1,2,g+1}$ be the stack whose objects over
$S$ consists of a pair $(P \stackrel{p} \to S, {\mathcal L})$ together
with an isomorphism $\phi\colon (P, {\mathcal L})\to (\Pro^1_S,
{\mathcal O}_{\Pro^1}(-g-1))$. The action of $\GL_2:= \Aut(\Pro^1, {\mathcal O}(1))$ on the pair
$(\Pro^1, {\mathcal O}(-g-1))$ has kernel $\mu_{g-1}$ and by \cite[Theorem 4.1]{ArVi:04}
$\tilde{{\mathcal H}}_{1,2,g+1}$ is represented by the scheme ${\mathbb A}^{2g+3}
\smallsetminus \Delta_1$ and the forgetful map $\tilde{\mathcal
H}_{1,2,g+1} \to {\mathcal H}_{1,2,g+1}$ is a
$\GL_2/\mu_{g+1}$-torsor. 

The vector bundle ${\mathcal V}_g$ pulls
back to the vector bundle which assigns to the trivial family
$p_*\colon \Pro^1_S \to S$ the vector bundle
$p_*(\omega_{\Pro^1}^{\otimes g/2} \otimes {\mathcal O}_{\Pro^1}(g+1))$.
As noted in \cite[Equation (4.4)]{GoVi:06} the Euler sequence for the tangent bundle of $\Pro^1$ implies
$\omega_{\Pro^1}$ is $\GL_2$-equivariantly
isomorphic to the bundle ${\mathcal O}(-2) \otimes \det E$ where $E$ is the defining representation of $\GL_2$.
Thus,
$$p_*(\omega_{\Pro^1}^{\otimes g/2} \otimes {\mathcal O}_{\Pro^1}(g+1))
= (\det E)^{\otimes g/2} \otimes E.$$ Now $(\det E)^{\otimes g/2} \otimes E$
is the pullback of $E$ via the map $\alpha\colon \GL_2 \to \GL_2$, $A \mapsto
(\det A)^{g/2} A$. As noted in \cite{ArVi:04}  $\ker \alpha
= \mu_{g+1}$, so under the identification of
$\GL_2/\mu_{g+1}=\GL_2$, $(\det E)^{\otimes g/2} \otimes E$ corresponds
to the defining representation $E$.
\end{proof}

\def\cprime{$'$}

\end{document}